\documentclass[12pt,a4paper,twoside]{article}
  \usepackage{url}
  \usepackage{amsthm}
  \usepackage{amsmath}
  \usepackage{amsfonts}
  \usepackage{amssymb}
  \usepackage{amscd}
  \usepackage{authblk}
  \usepackage[mathscr]{eucal}
  \usepackage{mathtools} 
  \usepackage{enumitem}
  \usepackage{color}
  \usepackage{fancyhdr}
  \usepackage{hyperref}
  \usepackage[top=2.5cm, bottom=2.7cm,left=3.5cm, right=2.5cm]{geometry}
 \usepackage[noabbrev]{cleveref} 

  \numberwithin{equation}{section}

 \allowdisplaybreaks[1] 

  \theoremstyle{definition}  
   \newtheorem{defn}{Definition}[section]

   \newtheorem{rmk}[defn]{Remark}

  \theoremstyle{plain}  
   \newtheorem{thm}[defn]{Theorem}
   \newtheorem{lem}[defn]{Lemma}
   
   \newtheorem{cor}[defn]{Corollary}

  \theoremstyle{remark} 

 \newcommand{\numberthis}{\refstepcounter{equation}\tag{\theequation}} 

\fontsize{12}{14}

  \DeclareMathOperator{\spn}{span}

\newcommand{\emailaddress}[1]{\newline{\sf#1}}

\title{Real Normal Operators and Williamson's Normal Form}
\date{}

\author[1]{Tiju Cherian John}
\author[2]{B.V. Rajarama Bhat}

\affil[1, 2]{Indian Statistical Institute,  8th Mile Mysore Road, R.V. College Post, Bangalore-560059, India. \emailaddress{ tijucherian@gmail.com, bhat@isibang.ac.in}}

\begin{document}
\maketitle
 \begin{abstract}
  A simple proof is provided to show that any bounded normal operator on a real Hilbert space is orthogonally equivalent to its transpose(adjoint). 
  A structure theorem for invertible skew-symmetric operators, which is analogous to the finite dimensional situation is also proved using elementary techniques.  
  The second result is used to establish the main theorem of this article, which is a generalization of Williamson's normal form for bounded positive operators on infinite dimensional separable Hilbert spaces. This has applications in the study of infinite mode Gaussian states.
 \end{abstract}

\section{Introduction}
Williamson's  theorem characterizing positive definite
real matrices of even order is fairly well-known and is a very
useful result. We begin by stating this theorem, which we are going
to generalize. Let $n$ be a natural number. Suppose $A$ is a
strictly positive (hence invertible) real matrix of order $2n\times
2n$. Then there exists a symplectic matrix $L$ of order $2n\times 2n$
and a strictly positive diagonal matrix $P$ of order $n\times n$, such that
$$A= L^T
\begin{bmatrix}
  P &0\\
  0&P
\end{bmatrix}L,
$$  where $L^T$ denotes the transpose of $L.$ Here the matrix $P$ is uniquely determined up to permutation and the
diagonal entries are  known as symplectic eigenvalues of $A$.

Symplectic transformations appear naturally in several physical models in classical mechanics, quantum mechanics, optics etc. For example, the collection of $2n\times 2n$ symplectic matrices, $n\in \mathbb{N}$, arises as the group of linear transformations preserving the classical Poisson Brakets and the canonical commutation relations (CCR) among the $n$ pairs of position and momentum observables \cite{ADMS95}. A similar situation occurs in the infinite dimensional situation also, in the context of Shales's theorem on the Bogoliubov automorphisms of Weyl commutation relations\cite{Sha62,  BhSr}.  In these scenarios, symplectic transformations play the role similar to that played by unitary or orthogonal transformations in Euclidean geometry. Thus the importance of Williamson's normal form in these contexts is similar to the importance of spectral theorem in linear algebra.

The  theorem mentioned in the first paragraph was first proved by J. Williamson in \cite{Wil36}. It
has been extensively used to understand the symplectic geometry and
has applications in Harmonic Analysis and  Physics (See \cite{Gos11}).
 In recent years there is somewhat renewed interest in the field \cite{BhJa15,IdGaWo17} in
view of its relevance in quantum information theory and its
usefulness to understand symmetries of finite mode quantum Gaussian
states \cite{Par13}. We wish to extend some of the results of Parthasarathy 
\cite{Par13}, to infinite mode Gaussian systems  (see \cite{BhJoSr18}) and for the purpose we need a generalization of the
Williamson's theorem to separable infinite dimensional Hilbert
spaces. Of course, such a generalization would be of independent
interest. Here we are precisely in the same situation as B\"ottcher and Pietsch \cite{Bot12} regarding operator theory on real Hilbert spaces, namely we are unable to find any source in literature where such a theorem has been worked out. The main goal of the present paper is to fill this apparent gap. In the process we find shortcuts and simplifications of some known results on real normal operators.

In this subject it is necessary  to deal with real linear operators on real
Hilbert spaces and their complexifications. We know that non-real eigenvalues of real matrices
appear in conjugate pairs. We need appropriate generalization of
this result in infinite dimensions. This is realised in Theorem \ref{thm:A similar A^T}, where we show that
every normal operator on a real Hilbert space is orthogonally
equivalent to its adjoint. This result is known \cite{Vis78}, however we
have an elementary direct proof of this fact. A more delicate result characterizing complex matrices unitarily equivalent to their adjoints can be seen in \cite{Gaja12}. The spectral theorem is a major tool in understanding real normal operators and  there is substantial literature on this, as can be seen from the following papers and references there of:  Wong \cite{Wong69}, Goodrich \cite{Goo72}, Viswanath \cite{Vis78}, Agrawal and Kulkarni \cite{AgKu94}. 
We prove an infinite dimensional version of Williamson's theorem, using Theorem \ref{cor:invertible} and some elements  of spectral theorem. This helps us to define symplectic spectrum for positive invertible operators on real Hilbert spaces. 

\section{Preliminary definitions and observations}
  In this Section, we shall collect the basic definitions and observations relevant to our work on real Hilbert spaces. To be consistent with the existing literature, we keep them similar to what is seen in \cite{Wong69} and \cite{Goo72}.
\begin{defn}
Let $A$ be a bounded operator on a real Hilbert space $(H,\left\langle \cdot,\cdot  \right\rangle)$. Its transpose $A^{T}$, is defined by $\left\langle Ax,y \right\rangle = \left\langle x,A^{T}y \right\rangle, \forall x,y \in H$. Such an $A^{T}$ exists uniquely as a bounded operator on $H$. $A$ is said to be normal if $AA^{T} = A^{T}A$.
\end{defn}
Following standard notation, for complex linear operators on complex Hilbert spaces  star ( $^*$ )  would denote the adjoint. Inner products on our complex Hilbert spaces would be anti-linear in the first variable. 

\begin{defn}
  Let $H$ be a real Hilbert space. Then the complexification of $H$ is the complex Hilbert space $\mathcal{H} :=H+iH:= \{x+i \cdot y:x\in H,y\in H\}$ with  addition, complex-scalar product and inner product defined in the obvious way.
\end{defn}

For example, if $\left\langle \cdot,\cdot \right\rangle$ is the inner product on $H$ then the inner product on $\mathcal{H}$ is given by
 $\left\langle x_{1} + i\cdot y_{1}, x_{2} + i \cdot y_{2} \right\rangle_{\mathbb{C}} := \left\langle x_{1} ,x_{2} \right\rangle + \left\langle y_{1},y_{2} \right\rangle + i \left( \left\langle x_{1},y_{2} \right\rangle - \left\langle y_{1},x_{2} \right\rangle \right)$. Also note that the mapping $x \mapsto x+i \cdot0$ provides an embedding of $H$ into $\mathcal{H}$ as a real Hilbert space.

\begin{defn}\label{def:spectrum}
  Let $A$ be a bounded operator on the real Hilbert space $H$. Define an operator $\hat{A}$ on the complexification $\mathcal{H}$ of
  $H$ by $\hat{A} (x+iy) = Ax +iAy$. Then $\hat{A}$ is well defined, complex linear and bounded, with $({\hat{A})}^* (x+iy) = A^Tx +i\cdot A^Ty = \widehat{(A^T)}(x+iy)$ and $\|\hat{A}\| = \|A\|$. If $A$ is normal then $\hat{A}$ is also normal. $\hat{A}$ is called the complexification of $A$.
Define the spectrum of $A$, denoted by $\sigma(A)$, to be the
spectrum of $\hat{A}.$
\end{defn}
Note that the  definition above of spectrum matches with  the usual
notion of eigenvalues of a finite dimensional real matrix.

\begin{defn}\label{def:cyclic-vec}
   Let $A$ be a bounded normal operator on a real Hilbert space $H$. Then a vector $x\in H$ is is said to be  {\em transpose cyclic\/} for $A$, if
the  set $\{ A^n(A^T)^mx: m, n\geq 0\}$ is total in $H.$
It is said to by {\em cyclic\/} for $A$,  if $\{ A^nx: n\geq 0\}$ is total
in $H.$
\end{defn}

\section {Symmetry of a real normal operator}
Here we prove that any normal operator on a real Hilbert space is orthogonally equivalent to its transpose (or adjoint). During the preparation of this manuscript we found that this result has appeared as Corollary 2.7 in Vishwanath \cite{Vis78}, but our proof is different and is very elementary. It just exploits the geometry of real Hilbert space. 
\begin{thm} \label{thm:A similar A^T}
Let $H$ be a real Hilbert space and let $A\in B(H)$ be a normal operator. Then there exists an orthogonal transformation $U \in B(H)$, such that
\begin{equation} \label{eq: A similar A^T}
UAU^{T} = A^{T}.
\end{equation}
Further,  $U$ can be chosen such that $U^T = U$.
\end{thm}
\begin{proof}
Let us assume first that $A$ has a transpose cyclic vector \textit{i.e.}
there exists $x \in H$ such that $\mathcal{E} := \{
A^n({A^T})^mx: n, m \geq 0\}$ is total in $H$. Define $U$ on
$\mathcal{E}$ by
\begin{equation}
U(A^n({A^T})^mx) = ({A^T})^nA^mx, ~\textnormal{for }  n,m\geq 0.
\end{equation}
Then for $n,m,k,l\geq 0,$
\begin{align*}
               \langle A^n({A^T})^mx, A^k({A^T})^lx \rangle
               &= \langle A^l({A^T})^kA^n({A^T})^mx, x \rangle \\
               &= \langle A^n({A^T})^m({A^T})^kA^lx, x \rangle \\
               &= \langle ({A^T})^kA^lx, ({A^T})^nA^mx \rangle \\
               &= \langle ({A^T})^nA^mx, ({A^T})^kA^lx \rangle \numberthis \label{eq:inner product} \\
               &= \langle U(A^n({A^T})^mx), U(A^k({A^T})^lx) \rangle , 
               \end{align*}
where the second equality follows from normality of $A$ and fourth equality because real inner product is symmetric. Since $U$ preserves the inner product on a total set $U$ can be extended as a bounded linear operator on $\overline{\textnormal{span}}\hspace{0.1cm} \mathcal{E}= H$. Note that the extended operator also preserves the inner product. We use the same symbol $U$ for the extended operator also. Thus $U$ is a real orthogonal transformation on $H$. Further using the definition of $U$,

note that
\begin{equation}
\langle U (A^m({A^T})^nx), A^k({A^T})^lx \rangle 
=\langle A^m({A^T})^nx, U(A^k({A^T})^lx) \rangle .
\end{equation}
Therefore,
\begin{equation} \label{eq: U sa}
 U^T=U.
 \end{equation}
Also,
\begin{align*}
             UAU^T (A^n({A^T})^mx) &= UAU(A^n({A^T})^mx) \\
                                 &= UA(({A^T})^nA^mx) \\
                                 &= U(A^{m+1}({A^T})^nx) \\
                                 &= ({A^T})^{m+1}A^nx \\
                                 &= A^T(A^n({A^T})^mx).
\end{align*}
Thus (\ref{eq: A similar A^T}) is satisfied on a total set which in turn proves the required relation on $H$, in the special case when $A$ has a transpose cyclic vector. The general case follows by a familiar application of Zorn's lemma.
\end{proof}

\begin{cor}\label{cor:unitary-equi}
$\hat{A}$ is unitarily equivalent to $(\hat{A})^{*} =\widehat{(A^T)}.$
\end{cor}

\begin{proof}
Let $U$ be as in Theorem \ref{thm:A similar A^T}, then $\hat{U}$ is a unitary which does the job.
\end{proof}

\begin{cor} \label{cor:spectr-symmetry} For any real normal operator $A$,  
 $\sigma(A) =  \sigma(A^{T}) =\overline{\sigma(A)}$ and thus the spectrum is symmetric about the real axis.
\end{cor}

\begin{proof}
Immediate from Definition \ref{def:spectrum} and Corollary  \ref{cor:unitary-equi}
\end{proof}

 Note that Corollary \ref{cor:spectr-symmetry} is analogous to the fact that complex eigenvalues of a real matrix occur in pairs.

\begin{thm}\label{cor:invertible}
  If $B$ is a skew symmetric invertible operator on a real Hilbert space $H$, then there exists a real Hilbert space $K$, a positive invertible operator $P$ on $K$ and a real orthogonal transformation $V \colon H \to K \oplus K $ such that

  \begin{equation}
    \label{eq:skew}
    B = V^{T} \begin{bmatrix}
               0   & -P  \\
               P   & 0
             \end{bmatrix} V.
  \end{equation}
\end{thm}
\begin{proof}
  Assume first that $B$ has a cyclic vector $x$.
 Note that if $B$ is skew-symmetric, $B^{2k+1}$ is skew-symmetric and $B^{2k}$ is symmetric for $k\in \mathbb{N}$. Therefore  $\langle x, B^{2k+1}x\rangle = 0$ for all $k\geq 0$  and moreover,
 \begin{equation}\label{eq:ip}
 \langle B^{2n}x, B^{2m+1}x\rangle = 0, \phantom{...}\forall n,m \geq 0, 
 \end{equation} 
 Set $K = \overline{\spn}\{x, B^2x,B^4x, \dots\}$ and $N =  \overline{\spn}\{Bx, B^3x,B^5x, \dots\}$. Then $K \perp N$ by (\ref{eq:ip}) and since $x$ is cyclic $N=K^{\perp}$. Therefore, there exists an operator $R: K\rightarrow N$ defined  by 
 $R := B|_{K}$. Then it is easily seen that,
 \begin{equation}\label{eq:A}
  B  = \begin{bmatrix}
   0 & -R^{T}\\
   R & 0
  \end{bmatrix},
  \end{equation}
  in the direct sum decomposition $H= K\oplus N$. Since $B(B^{2n}x) = B^{2n+1}x$, $R$ maps $K$ onto $N$ and  since $B$ is invertible  $R$ is an invertible operator. Now we apply polar decomposition to $R$. If $R= UP$ then $U:K\rightarrow N$ and $P:K\rightarrow K$ are such that $U$ is orthogonal (because $R$ is invertible) and $P$ ($=\sqrt[]{R^{T}R}$) is positive definite and invertible. 
 
  Now we have \begin{equation}
 B = \begin{bmatrix}
  I_K & 0\\
  0 & U
 \end{bmatrix} \begin{bmatrix}
   0 & -P\\
   P & 0
  \end{bmatrix} \begin{bmatrix}
  I_K & 0\\
  0 & U^{\tau}
  \end{bmatrix}, 
  \end{equation}
  where $I_K $ is the identity operator on $K$ and $\begin{bsmallmatrix}
  I_K & 0\\
  0 & U
  \end{bsmallmatrix} : K\oplus K \rightarrow K\oplus N$ is orthogonal. 
 
  If $B$ doesn't have a cyclic vector then a usual argument using Zorn's lemma along with a permutation proves the result.
\end{proof}

\section{Williamson's Normal Form}

We will use Theorem \ref{cor:invertible} to give a proof of Williamson's normal form in the infinite dimensional set up.  We refer to Parthasarathy \cite{Par13} for an easy proof of this theorem in the finite dimensional setup and we extend it.

\begin{defn}\label{def:J-operator}
  Let $H$ be a real Hilbert space and $I$ be the identity operator on $H$. Define the involution operator $J$ on $H \oplus H$ by
$J = \begin{bmatrix}
         0 & -I\\
         I & 0
     \end{bmatrix}.$
\end{defn}

\begin{defn}
 Let $H$ and $K$ be two real Hilbert spaces. A bounded invertible linear operator $Q \colon H\oplus H \to K\oplus K $  is called a symplectic transformation if $Q^{T}JQ = J$, where $J$ on left side is the involution operator on $K\oplus K$ and that on the right side it is the involution operator on $ H\oplus H$.
\end{defn}

Here is the main Theorem.

\begin{thm}\label{sec:will-norm-form}
  Let $H$ be a real Hilbert space and $A$ be a strictly positive invertible operator on $H\oplus H$ then there exists

a Hilbert space $K$, a positive invertible operator $P$ on $K$ and a symplectic transformation $ L\colon H\oplus H \rightarrow K \oplus K$

 such that   \begin{equation}
    \label{eq:will-norm-form}
    A = L^{T} \begin{bmatrix}
              P & 0  \\
              0 & P
             \end{bmatrix} L.
  \end{equation}

The decomposition is unique in the sense that if $M$ is any strictly positive invertible operator on a Hilbert space $\tilde{H}$ and $\tilde{L}\colon H\oplus H \rightarrow \tilde{H} \oplus \tilde{H}$ is a symplectic transformation such that
 \begin{equation}\label{eq:uniqueness}
    A = \tilde{L}^{T} \begin{bmatrix}
              M & 0  \\
              0 & M
             \end{bmatrix} \tilde{L},
  \end{equation}
then

$P$ and $M$ are orthogonally equivalent.
\end{thm}

\begin{proof}
  Define $B= A^{1/2}JA^{1/2}$,
  and let $J$ be the involution as in Definition \ref{def:J-operator}. Then $B$ is a skew symmetric invertible operator on $H\oplus H$.
 Hence by Theorem \ref{cor:invertible}  there exists a real Hilbert space $K$, an invertible positive operator $P$ and  a real orthogonal transformation $V \colon H \to K\oplus K $ such that
\begin{equation} \label{eq:B}
    B  = V^T \begin{bmatrix}
                          0 & -P   \\
                          P & 0
                         \end{bmatrix}V.
  \end{equation}

Define $L \colon  H\oplus H \to K\oplus K ,$ by 
 \begin{equation}
   L      = \begin{bmatrix}
                          P^{-1/2}  & 0  \\
                           0 &  P^{-1/2}
                         \end{bmatrix} VA^{\frac{1}{2}}.
  \end{equation}

Then clearly (\ref{eq:will-norm-form}) is satisfied.  A direct computation using (\ref{eq:B}) shows that $L$ is symplectic, that is $
LJL^T = J$, where $J$ on the left side  is the involution operator on $H\oplus H$ and on the right side is the corresponding involution operator on $K\oplus K.$  

To prove the uniqueness,  let 

\begin{equation*}
    A =L^{T}\begin{bmatrix}
              P & 0  \\
               0  & P
             \end{bmatrix} L
      = \tilde{L}^{T} \begin{bmatrix}
              M & 0  \\
              0 & M
             \end{bmatrix} \tilde{L},
  \end{equation*} where $P$,  $M$ are two positive operators and $L$, $\tilde{L}$ are symplectic.
 Putting $N=L\tilde{L}^{-1}$ we get a symplectic $N$ such that
\begin{equation*}
   N^{T} \begin{bmatrix}
              P & 0  \\
               0                     & P
             \end{bmatrix} N
      =  \begin{bmatrix}
              M & 0  \\
              0 & M
             \end{bmatrix}.
\end{equation*}
Substituting $N^{T} = JN^{-1}J^{-1}$ with appropriate $J$'s we get

\begin{equation}\label{eq:2}
  N^{-1}\begin{bmatrix}
              0 &  P   \\
               - P & 0
             \end{bmatrix} N
    = \begin{bmatrix}
              0 & M  \\
              -M & 0
             \end{bmatrix}.
\end{equation}

We wish to replace the $N$ in previous equation with an orthogonal map. Now we may recall the fact that two similar normal operators are unitarily equivalent (this can be proved using Fuglede-Putnam theorem, see Theorem 12.36 in \cite{Rud91} and real case follows by complexification). However, we continue with our proof without using this result. To this end, taking transpose on both sides of (\ref{eq:2}) we get

 \begin{equation*}
     N^{T}\begin{bmatrix}
              0 &  P   \\
               - P & 0
             \end{bmatrix} (N^{T})^{-1}
    = \begin{bmatrix}
              0 & M  \\
              -M & 0
             \end{bmatrix}.
 \end{equation*}
Hence again by using (\ref{eq:2}) we get
\begin{equation*}
  N^{T}\begin{bmatrix}
              0 &  P   \\
               - P & 0
             \end{bmatrix} (N^{T})^{-1}
  =  N^{-1}\begin{bmatrix}
              0 &  P   \\
               - P & 0
             \end{bmatrix} N,
\end{equation*}
  or
\begin{equation*}
  \begin{bmatrix}
              0 &  P   \\
               - P & 0
             \end{bmatrix} (N^{-1})^{T}N^{-1}
  =  (N^{-1})^{T}N^{-1}\begin{bmatrix}
              0 &  P   \\
               - P & 0
             \end{bmatrix}.
\end{equation*}
This implies
\begin{equation}\label{eq:1}
   \begin{bmatrix}
              0 &  P   \\
               - P & 0
             \end{bmatrix} \big((N^{-1})^{T}N^{-1}\big)^{1/2}
  = \big( (N^{-1})^{T}N^{-1}\big)^{1/2} \begin{bmatrix}
              0 &  P   \\
               - P & 0
             \end{bmatrix},
\end{equation}
where the reasoning for (\ref{eq:1}) is same as that in the complex case. Let $N^{-1} = U\big( (N^{-1})^{T}N^{-1}\big)^{1/2}$ be the polar decomposition of $N^{-1}$. From (\ref{eq:2}) we get
\begin{equation*}
  U\big( (N^{-1})^{T}N^{-1}\big)^{1/2}\begin{bmatrix}
              0 &  P   \\
               - P & 0
             \end{bmatrix}\big( (N^{-1})^{T}N^{-1}\big)^{-1/2} U^{T}
    = \begin{bmatrix}
              0 & M  \\
              -M & 0
             \end{bmatrix}.
\end{equation*}
Hence by (\ref{eq:1}) we have
\begin{equation}\label{eqW:skew-unitary2}
    U\begin{bmatrix}
              0 &  P   \\
               - P & 0
             \end{bmatrix}U^{T}
    = \begin{bmatrix}
              0 & M  \\
              -M & 0
             \end{bmatrix}.
\end{equation}
Now we will prove that (\ref{eqW:skew-unitary2}) implies  that $P$ and $M$ are orthogonally equivalent. Note that by taking squares, and getting rid of the negative sign,
\begin{equation*}
              U \begin{bmatrix}
              P^2 &  0   \\
                 0 & P^2
             \end{bmatrix}U^{T} =\begin{bmatrix}
              M^2 &  0   \\
                 0 &  M^2
             \end{bmatrix}.
\end{equation*} It is true that if $A$ and $B$ are self adjoint operators such that $A\oplus A$ and $B \oplus B$ are orthogonally equivalent then $A$ and $B$ are orthogonally equivalent. We will give a proof of this as a Lemma below. But if we assume this fact our proof is complete because we see that for the positive operators $P$ and $M$, $P^2$ and $M^2$ are orthogonally equivalent. Hence $P$ and $M$ are orthogonally equivalent.
\end{proof}

 Now we will prove an important corollary of the above theorem. It states that the Hilbert space $K$ in Theorem \ref{sec:will-norm-form} can be chosen to be $H$ itself. 
\begin{cor}\label{sec:will-norm-form-1}
Let $H$ be a real Hilbert space and $A$ be a strictly positive invertible operator on $H\oplus H$ then there exists a positive invertible operator $D$ on $H$ and a symplectic transformation $ M\colon H\oplus H \rightarrow H \oplus H$
 such that   \begin{equation}
    \label{eq:will-norm-form-1}
    A = M^{T} \begin{bmatrix}
              D & 0  \\
              0 & D
             \end{bmatrix} M.
  \end{equation}
  Further, $D$ is unique up to a conjugation with an orthogonal transformation.
\end{cor}
\begin{proof}
Let $L$ and $P$ be as in Theorem \ref{sec:will-norm-form}. Since $L$ is invertible $H$ and $K$ are of same dimension. Choose and fix any orthogonal transformation $U_0: K \rightarrow H$. Then $U = \begin{bmatrix}U_0  & 0 \\ 0 & U_0\end{bmatrix}$ is a symplectic (and orthogonal) transformation. Now by (\ref{eq:will-norm-form}) 
\begin{align*}
    A &= L^{T} \begin{bmatrix}
              P & 0  \\
              0 & P
             \end{bmatrix} L\\
             &= L^{T}U^TU \begin{bmatrix}
              P & 0  \\
              0 & P
             \end{bmatrix}U^TU L\\
             &= M^T\begin{bmatrix}
              D & 0  \\
              0 & D
             \end{bmatrix}M,
\end{align*}
where  $M:= UL$ is symplectic on $H\oplus H$ and $D:= U_0PU_0^T$ is a strictly positive and invertible operator on $H$.
\end{proof}
Now we proceed to provide the proof of a lemma we promised during the proof of Theorem \ref{sec:will-norm-form}. We depend on Hall \cite{Hall13} for notations and results used below. We write the following in the framework of complex Hilbert spaces, but the spectral theory of a self-adjoint operator is identical on both real and complex Hilbert spaces \cite{RiN90} and hence what we write below works on separable real Hilbert spaces also.

By the direct integral version of spectral theorem, any bounded self-adjoint operator $A$ on a separable Hilbert space is unitarily equivalent to the multiplication operator 
 $s\mapsto xs$  where $xs(\lambda) := \lambda s(\lambda ), \lambda \in \sigma(A)$ on $\int_{\sigma(A)}^{\oplus}\mathcal{H}_{\lambda}d\,\mu(\lambda)$ for some $\sigma$-finite measure $\mu$ with a measurable family  of Hilbert spaces $\{\mathcal{H}_{\lambda}\}$, satisfying dim$( \mathcal{H}_{\lambda})>0   $ almost everywhere $\mu .$ It is understood that we work with the Borel subsets of the spectrum $\sigma(A)$.  The function  $\lambda \mapsto \dim (\mathcal{H}_{\lambda})$ is called the multiplicity function associated with the direct integral representation of $A$.
By Proposition 7.24 from \cite{Hall13}, two bounded self-adjoint operators expressed as direct integrals on their spectrum are unitarily equivalent if and only if (i) the spectrum are same; (ii) the associated measures are equivalent in the sense that they are mutually absolutely continuous and (iii) the multiplicity functions coincide almost everywhere.

\begin{lem}\label{lem:direct-sum-unitary}
Let $A, B$ be self-adjoint operators on separable Hilbert space such that $A \oplus A$ and $B \oplus B$ are unitarily equivalent. Then $A$ and $B$ are unitarily equivalent.
\end{lem}
\begin{proof}
Let $A$ be unitarily equivalent to the multiplication operator 
for each section $s$, with respect to a measure $\mu$ on $\sigma (A)$ in the direct integral Hilbert space 
\[\int_{\sigma(A)}^{\oplus}\mathcal{H}_{\lambda}d\,\mu(\lambda).\] Then it can be seen that $A \oplus A$ is unitarily equivalent to the multiplication operator on the direct integral 
\[\int_{\sigma(A)}^{\oplus}\mathcal{K}_{\lambda}d\,\mu(\lambda) , \]  where 
$\mathcal{K}_{\lambda} = \mathcal{H}_{\lambda} \oplus \mathcal{H}_{\lambda}.$ Since $A \oplus A$ and $B \oplus B$ are unitarily equivalent, by the uniqueness of integral representation mentioned above,  by comparing spectrum, measures and multiplicity functions it is easy to see that $A$ and $B$ are unitarily equivalent.
\end{proof}

\begin{rmk}
We observe that Lemma \ref{lem:direct-sum-unitary} can be proved using standard versions of Hahn-Hellinger theorem also, for example Theorem 7.6 in \cite{Par12} can also be used. We also note that if we take infinitely many copies of self-adjoint operators $A,B$ and $\oplus _{i=1}^{\infty} A$ is unitarily equivalent to 
$\oplus _{i=1}^{\infty} B$, does not mean that $A$ and $B$ are unitarily equivalent. So it is only natural that the multiplicity theory is required in the proof of last Lemma. The Lemma \ref{lem:direct-sum-unitary} is true even without the assumption of self-adjointness of $A$ and $B$ as seen by \cite{kadison1957}.
\end{rmk}

\begin{rmk}
Under the situation of Theorem 5.3, in view of the uniqueness part of the theorem, the spectrum of $P$, can be defined as the {\em symplectic spectrum\/}  of the positive invertible operator $A$.
\end{rmk}
\textbf{Conclusion:} 
As suggested in the introduction of this article, finite dimensional Williamson's normal form is widely used in several areas.  It is expected that the main theorem of this article will be useful in all these areas when one wants to discuss infinite dimensional analogues. One particular case is \cite{BhJoSr18},  where non-commutative analogues of classical Gaussian distributions called quantum Gaussian states are studied.  
Strictly positive, invertible operators occur as covariance operators of quantum Gaussian states. Theorem \ref{sec:will-norm-form} and Corollary \ref{sec:will-norm-form-1} are repeatedly used in \cite{BhJoSr18} to obtain various results about covariance operators and Gaussian states. One of the main results there is that the symplectic spectrum, which was defined as a consequence of our main theorem, is a complete invariant for the class of covariance operators associated with quantum Gaussian states.  It is clear that the Williamson's normal form in infinite dimension  is indispensable  in this context. 

\textbf{Acknowledgements:} We wish to thank Prof. K. R. Parthasarathy for introducing the subject to us during his 2015 lecture series at the Indian Statistical Institute, Bangalore centre. We also thank Prof. K. B. Sinha for the fruitful discussions we had with him on direct integrals. We are grateful to the anonymous referee  for a positive report and very constructive suggestions. Bhat thanks J C Bose Fellowship for financial support.

\bibliographystyle{amsalpha}
\bibliography{bibliography}

\providecommand{\bysame}{\leavevmode\hbox to3em{\hrulefill}\thinspace}
\providecommand{\MR}{\relax\ifhmode\unskip\space\fi MR }
\providecommand{\MRhref}[2]{%
  \href{http://www.ams.org/mathscinet-getitem?mr=#1}{#2}
}
\providecommand{\href}[2]{#2}
\begin{thebibliography}{ADMS95}

\bibitem[ADMS95]{ADMS95}
Arvind, B.~Dutta, N.~Mukunda, and R.~Simon, \emph{The real symplectic groups in
  quantum mechanics and optics}, Pramana \textbf{45} (1995), no.~6, 471--497.

\bibitem[AK94]{AgKu94}
Sushama~N. Agrawal and S.~H. Kulkarni, \emph{A spectral theorem for a normal
  operator on a real {H}ilbert space}, Acta Sci. Math. (Szeged) \textbf{59}
  (1994), no.~3-4, 441--451. \MR{1317165}

\bibitem[BJ15]{BhJa15}
Rajendra Bhatia and Tanvi Jain, \emph{On symplectic eigenvalues of positive
  definite matrices}, J. Math. Phys. \textbf{56} (2015), no.~11, 112201, 16.
  \MR{3425187}

\bibitem[BJS18]{BhJoSr18}
B.~V.~Rajarama Bhat, Tiju~Cherian John, and R.~Srinivasan, \emph{{Infinite Mode
  Quantum Gaussian States}}, ArXiv e-prints, 1804.05049 (2018).

\bibitem[BP12]{Bot12}
Albrecht B\"ottcher and Albrecht Pietsch, \emph{Orthogonal and skew-symmetric
  operators in real {H}ilbert space}, Integral Equations Operator Theory
  \textbf{74} (2012), no.~4, 497--511. \MR{3000432}

\bibitem[BS05]{BhSr}
B.~V.~Rajarama Bhat and R.~Srinivasan, \emph{On product systems arising from
  sum systems}, Infin. Dimens. Anal. Quantum Probab. Relat. Top. \textbf{8}
  (2005), no.~1, 1--31. \MR{2126876}

\bibitem[dG11]{Gos11}
Maurice~A. de~Gosson, \emph{Symplectic methods in harmonic analysis and in
  mathematical physics}, Pseudo-Differential Operators. Theory and
  Applications, vol.~7, Birkh\"auser/Springer Basel AG, Basel, 2011.
  \MR{2827662}

\bibitem[Goo72]{Goo72}
Robert~Kent Goodrich, \emph{The spectral theorem for real {H}ilbert space},
  Acta Sci. Math. (Szeged) \textbf{33} (1972), 123--127. \MR{0305123}

\bibitem[GT12]{Gaja12}
Stephan~Ramon Garcia and James~E. Tener, \emph{Unitary equivalence of a matrix
  to its transpose}, J. Operator Theory \textbf{68} (2012), no.~1, 179--203.
  \MR{2966041}

\bibitem[Hal13]{Hall13}
Brian~C. Hall, \emph{Quantum theory for mathematicians}, Graduate Texts in
  Mathematics, vol. 267, Springer, New York, 2013. \MR{3112817}

\bibitem[ISGW17]{IdGaWo17}
Martin Idel, Sebasti\'an Soto~Gaona, and Michael~M. Wolf, \emph{Perturbation
  bounds for {W}illiamson's symplectic normal form}, Linear Algebra Appl.
  \textbf{525} (2017), 45--58. \MR{3632235}

\bibitem[KS57]{kadison1957}
Richard~V. Kadison and I.~M. Singer, \emph{Three test problems in operator
  theory.}, Pacific J. Math. \textbf{7} (1957), no.~2, 1101--1106.

\bibitem[Par92]{Par12}
K.~R. Parthasarathy, \emph{An introduction to quantum stochastic calculus},
  Modern Birkh\"auser Classics, Birkh\"auser/Springer Basel AG, Basel, 1992,
  [2012 reprint of the 1992 original] [MR1164866]. \MR{3012668}

\bibitem[Par13]{Par13}
Kalyanapuram~R. Parthasarathy, \emph{The symmetry group of {G}aussian states in
  {$L^2(\mathbb{R}^n)$}}, Prokhorov and contemporary probability theory,
  Springer Proc. Math. Stat., vol.~33, Springer, Heidelberg, 2013,
  pp.~349--369. \MR{3070484}

\bibitem[RSN90]{RiN90}
Frigyes Riesz and B{\'e}la Sz.-Nagy, \emph{Functional analysis}, Dover Books on
  Advanced Mathematics, Dover Publications, Inc., New York, 1990, Translated
  from the second French edition by Leo F. Boron, Reprint of the 1955 original.
  \MR{1068530}

\bibitem[Rud91]{Rud91}
Walter Rudin, \emph{Functional analysis}, second ed., International Series in
  Pure and Applied Mathematics, McGraw-Hill, Inc., New York, 1991. \MR{1157815}

\bibitem[Sha62]{Sha62}
David Shale, \emph{Linear symmetries of free boson fields}, Trans. Amer. Math.
  Soc. \textbf{103} (1962), 149--167. \MR{0137504}

\bibitem[Vis78]{Vis78}
K.~Viswanath, \emph{Operators on real {H}ilbert spaces}, J. Indian Math. Soc.
  (N.S.) \textbf{42} (1978), no.~1-4, 1--13 (1979). \MR{558979}

\bibitem[Wil36]{Wil36}
John Williamson, \emph{On the {A}lgebraic {P}roblem {C}oncerning the {N}ormal
  {F}orms of {L}inear {D}ynamical {S}ystems}, Amer. J. Math. \textbf{58}
  (1936), no.~1, 141--163. \MR{1507138}

\bibitem[Won69]{Wong69}
Tin~Kin Wong, \emph{On real normal operators}, ProQuest LLC, Ann Arbor, MI,
  1969, Thesis (Ph.D.)--Indiana University. \MR{2618858}

\end{thebibliography}

\end{document}